\documentclass[12pt]{amsart}

\usepackage{amsmath,amssymb,amsthm,graphicx}
\usepackage{amsfonts}
\usepackage[mathscr]{eucal}
\usepackage{color}

\newtheorem{theorem}{Theorem}[section]
\newtheorem{lemma}[theorem]{Lemma}
\newtheorem{prop}[theorem]{Proposition}
\newtheorem{cor}[theorem]{Corollary}
\newtheorem{corollary}[theorem]{Corollary}

\theoremstyle{definition}

\theoremstyle{remark}
\newtheorem{remark}[theorem]{Remark}

\numberwithin{equation}{section}

\newcommand{\be}{\begin{equation}}
\newcommand{\ee}{\end{equation}}

\newcommand{\ep}{\varepsilon}

\def\dimh{{\rm dim}_{_{\rm H}}}
\def\dimp{{\rm dim}_{_{\rm P}}}
\def\dimm{\overline{\dim}_{_{\rm M}}}
\def\R{{\mathbb R}}

\def\Q{{\mathbb Q}}

\def\l{{\langle}}
\def\r{\rangle}

\def\a{\alpha}

\def\ga{\gamma}

\def\ep{\varepsilon}
\def\eps{\varepsilon}
\def\si{\sigma}
\def\Re {{\rm Re}\,}

\def\E{{\mathbb E}}
\def\P{{\mathbb P}}

\begin{document}

\title{\bf Fractal Dimensions for Continuous Time Random Walk Limits}

\author{Mark M. Meerschaert}
\address{Mark M. Meerschaert, Department of Probability and Statistics,
Michigan State University, East Lansing, MI 48824}
\email{mcubed@stt.msu.edu}
\urladdr{http://www.stt.msu.edu/$\sim$mcubed/}
\thanks{Research of M. M. Meerschaert was partially
supported by NSF grants DMS-0125486, DMS-0803360, EAR-0823965
and NIH grant R01-EB012079-01.}

\author{Erkan Nane}
\address{Erkan Nane, Department of Mathematics and Statistics,
Auburn University, Auburn, AL 36849}
\email{nane@auburn.edu}
\urladdr{http://www.auburn.edu/$\sim$ezn0001}

\author{Yimin Xiao}
\address{Yimin Xiao, Department Statistics and Probability,
Michigan State University, East Lansing, MI 48824}
\email{xiao@stt.msu.edu}
\urladdr{http://www.stt.msu.edu/$\sim$xiaoyimi}
\thanks{Research of Y. Xiao was partially supported by
NSF grant DMS-1006903.}

\date{\today}

\begin{abstract}
In a continuous time random walk (CTRW), each random jump follows
a random waiting time. CTRW scaling limits are time-changed processes
that model anomalous diffusion. The outer process describes particle
jumps, and the non-Markovian inner process (or time change) accounts
for waiting times between jumps. This paper studies fractal properties
of the sample functions of a time-changed process, and establishes some
general results on the Hausdorff and packing dimensions of its range
and graph. Then those results are applied to CTRW scaling limits.
\end{abstract}


%\textbf{Mathematics Subject Classification (2000):} 60J65, 60K99.
\maketitle
%\newpage
%\tableofcontents

\section{Introduction}

Continuous time random walks have attracted a lot of attention in
recent years. They provide flexible models for anomalous diffusion phenomena
in a wide range of scientific areas including physics, finance and hydrology.
Consider a random walk $S(n)=J_1+\cdots+J_n$ on $\R^d$, where $\{J_n, n \ge 1\}$
model the particle jumps. The continuous time random walk (CTRW) imposes
a random waiting time between jumps. Let $T_n=W_1+\cdots+W_n$, where $\{W_n, n \ge 1\}$
are nonnegative random variables. The CTRW jumps to location $S(n)$ at time $T_n$. The
number of jumps by time $t\geq 0$ is given by the counting process $N_t
=\max\{n\geq 0:T_n\leq t\}$, where $T_0=0$. The time-changed process
$S(N_t)$ represents the location of a random walker at time $t\geq 0$. A standard
assumption in the literature is that $\{(J_n, Y_n),\, n \ge 1\}$ are iid. In recent
years CTRW with dependent jumps or/and waiting times have also been considered, see
for example \cite{CHS,MNX09,TM10}.

The scaling limit of a CTRW $\{S(N_t), t \ge 0\}$ is a time-changed (or iterated)
process $X=\{X(t), t \ge 0\}$ of the form $X(t) = Y(E_t)$, where the outer process
$\{Y(t), t \ge 0\}$ is the scaling limit of the random walk $
\{S_n, n \ge 0\}$ and the inner process $\{E_t, t \ge 0\}$ accounts for
the random waiting times $\{W_n, n \ge 1\}$. This has been proved by Meerschaert and
Scheffler \cite{MS04}, and Becker-Kern,  Meerschaert and Scheffler \cite{coupleCTRW,CTRW}
under the assumption that $\{(J_n, W_n), n \ge 1\}$ are iid, the jumps
$\{J_n, n \ge 1\}$ belong to the domain of attraction of an operator stable
law and the waiting times $\{W_n, n \ge 1\}$ belong to the strict domain
of attraction of a positive stable random variable $D$ with index $\beta \in (0, 1)$.
In this case, the outer process $\{Y(t), t \ge 0\}$ is an operator stable
L\'evy process with values in $\R^d$ and the inner process $\{(E_t, t \ge 0\}$
is the inverse of a $\beta$-stable subordinator $\{D(x), x \ge 0\}$ with $D(1) = D$.
Namely,
\be\label{Etdef}
E_t = \inf\{x\ge 0: D(x) > t\}, \quad \forall\ t \ge 0.
\ee
The aforementioned authors further proved that the density function $p(t, x)$ of
$X(t) =  Y(E_t)$ solves fractional partial differential equations; see
\cite{fracCauchy,bmn-07,Zsolution} and the references therein for further
information on PDE connections of CTRW limits.

When the independence assumption on the jumps $\{J_n, n \ge 1\}$ is removed,
Meerschaert, Nane and Xiao \cite{MNX09} showed that the outer process $Y$ can
be taken as a fractional Brownian motion, a stable L\'evy process or a linear
fractional stable motion. More general inner processes may also be possible if
the waiting times are dependent.

\begin{figure}
\includegraphics[width=3.5in]{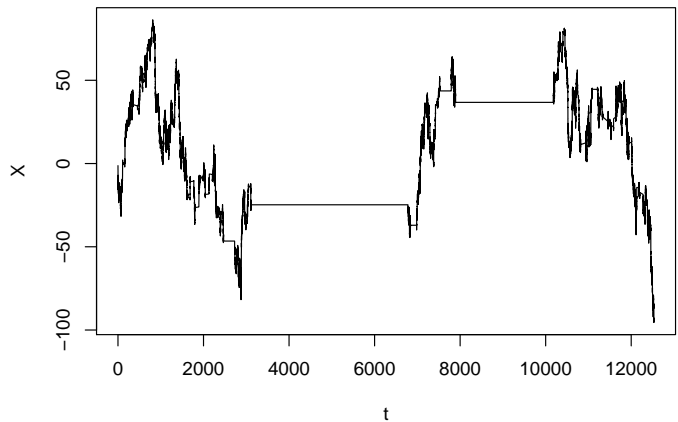}
 \caption{Typical sample path of the iterated process $X(t)=Y(E_t)$.
 Here $Y(t)$ is a Brownian motion and $E_t$ is the inverse of an $0.8$-stable subordinator.
 }\label{fig1}
\end{figure}

In general, a CTRW limiting process $X$ is non-Markovian, non-Gaussian and satisfies
a form of self-similarity. Figure \ref{fig1} illustrates a typical trajectory
of the time-changed process $X(t)=Y(E_t)$,
in the case where the outer process $Y$ is a Brownian motion. The graph resembles
that of a Brownian motion, interrupted by long resting periods. This process $X$ is
the long-time scaling limit of a CTRW with mean zero, finite variance jumps and heavy
tailed waiting times in the domain of attraction of a $\beta$-stable subordinator,
see \cite{MS04}.

This paper is concerned with fractal properties of the CTRW limiting process
$X =\{X(t), t \ge 0\}$ defined by $X(t)= Y(E_t)$ for $t\ge 0$. In particular we
determine the Hausdorff and packing dimensions of the range $X([0, 1]) =\{X(t):
t \in [ 0, 1]\}$ and the graph ${\rm Gr} X([0, 1])= \{(t, X(t)):  t \in[0,1]\}$.
There has been a large literature on sample path and fractal properties
of L\'evy processes \cite{T86,X04}, and Gaussian or stable random fields
\cite{K85,X09}. Several methods have been developed for computing the
Hausdorff dimensions of the range and graph of stochastic
processes under ``minimal'' conditions. To give a brief description of
the general method, let $U = \{U(t), t \ge 0\}$ be a stochastic
process with values in $\R^d$ (for simplicity we assume that the components
of $U(t)$ are independent). If there
exist positive constants $C$
and $H \in (0, 1)$ such that
\begin{equation}\label{Eq:Pre1}
\E\bigg(\sup_{0 \le h \le T}|U(t+h) - U(t)|\bigg) \le C T^H, \qquad
\forall T \in (0, 1),
\end{equation}
then one can prove
\begin{equation}\label{Eq:Pre2}
\dimh U([0, 1]) \le \min\bigg\{d, \frac 1 H\bigg\}\quad \hbox{a.s.}
\end{equation}
and
\begin{equation}\label{Eq:Pre3}
\dimh{\rm Gr} U([0, 1]) \le \min\bigg\{\frac 1 H, 1 + (1-H)d\bigg\}
\quad \hbox{a.s.}
\end{equation}
In the above and sequel, $\dimh $ denotes Hausdorff dimension. Moreover,
if there exist positive constants $C$
and $H \in (0, 1)$ such that
\begin{equation}\label{Eq:Pre4}
\P\bigg( |U(t) - U(s)| \le |t-s|^H x \bigg) \le C \min\big\{1, x^d\big\},
\quad \forall s, t \in [0, 1],
\end{equation}
then equalities hold in both (\ref{Eq:Pre2}) and (\ref{Eq:Pre3}). The above
method can be applied to a wide class of stochastic processes, including
self-similar processes with stationary increments such as stable L\'evy processes,
fractional Brownian motion and the iterated Brownian motion (\cite{burdzy1,burdzy2}).
See \cite{Fal91,XiaoLin,SX10} for further information.

However, the time-changed process $X=\{Y(E_t), t\ge 0\}$ considered in
this paper does not satisfy (\ref{Eq:Pre4}).  In fact, if we consider the process
$X(t) = Y(E_t)$ in Figure 1, where $Y$ is a Brownian motion in $\R^d$ and
$E_t$ is the inverse of a $\beta$-stable subordinator $D$ defined by (\ref{Etdef}),
then the inner process $E_t$ remains constant over infinitely many intervals,
corresponding to the jumps of the stable subordinator $D$. Hence the graph of the
process $X$ remains flat over these resting intervals, as evidenced by Figure 1.
More precisely, it can be proved by using Proposition 2 in
Chapter III of Bertoin \cite{Bertoin96} that, for any $s < t$,
$\P\big\{E_s = E_t\big\} > 0$. This implies that $\P\big\{Y(E_s) - Y(E_t) =
0\big\} > 0$. Hence $X$ does not satisfy (\ref{Eq:Pre4}). As we will see from
Propositions \ref{prop:Levy} and \ref{Prop:FBM}, the actual value of
$\dimh {\rm Gr} X([0, 1])$ may be strictly smaller than what is suggested by
(\ref{Eq:Pre3}).

Packing dimension was introduced in 1980's by Tricot \cite{Tricot} as a dual
concept to Hausdorff dimension, and has become a useful tool for analyzing fractal
sets and sample paths of stochastic processes. It is known that Hausdorff and
packing dimensions of a set $E$ characterize different geometric aspects of $E$
and many random fractals arising in studies of stochastic processes
have different Hausdorff and packing dimensions. See \cite{T86,X04} and the
references therein for more information and \cite{KX08,KSX10,SX10} for recent
development. A fractal set $E \subseteq \R^d$ with the property  $\dimh E =\dimp E$
is usually called a regular fractal. We will see that the range and graph of CTRW
limits considered in Section 3 are often regular fractals.

The rest of this paper is organized as follows. In Section 2 we prove under
quite general conditions that
\begin{equation}\label{dimX}
\dimh X([0, 1]) = \dimh Y([0, 1]) \ \hbox{ and }\  \dimp X([0, 1]) = \dimp Y([0, 1]),
\quad {\rm a.s.},
\end{equation}%\min \Big\{d, \frac 1 {H}\Big\}
\begin{equation}\label{dimGrX}
\dimh {\rm Gr} X([0, 1]) = \max\big\{1, \dimh Z([0, 1])\big\}   \quad
{\rm a.s.},
\end{equation}%\min \Big\{\frac 1 {H},\ 1 + (1- H) d\Big\}
and
\begin{equation}\label{DimGrX}
  \dimp {\rm Gr} X([0, 1]) = \max\big\{1, \dimp Z([0, 1])\big\}  \quad
{\rm a.s.},
\end{equation}%\min \Big\{\frac 1 {H},\ 1 + (1- H) d\Big\}
where $Z= \{Z(x), x \ge 0\}$ is the $\R^{d+1}$-valued process defined
by $Z(x)= (D(x), Y(x))$ (see (\ref{Eq:Z}) and (\ref{Eq:D}) below).
These results are applied in Section 3 to the scaling limits
of continuous time random walks.  First we consider the uncoupled case, in which
the iid waiting times $\{W_n, n \ge 1\}$ are independent of the iid particle
jumps $\{J_n, n \ge 1\}$.  Then we treat
certain coupled examples, where the jump depends on the previous waiting time.
We also consider triangular array CTRW limits, which lead to general inverse
subordinators. Finally we
examine the case of correlated jumps. In all these cases, the outer process $Y$
is either a L\'evy process or a fractional Brownian motion.

\section{General Results}

In this section we prove some general results on the Hausdorff and packing dimensions
of the range and graph of the time-changed process $X =\{X(t), t \ge 0\}$ defined
by $X(t)= Y(E_t)$ for $t\ge 0$.
We assume that $Y =\{Y(x), x \ge 0\}$ is a stochastic process with values
in $\R^d$ and $E = \{E_t, t \ge 0\}$ is a process with $E_0 = 0$ and nondecreasing
continuous sample functions. Both processes $Y$ and $E$ are defined on a
probability space $(\Omega, {\mathcal F}, \P)$, and they are not necessarily
independent.  In the following section, the results in this section will be
applied to CTRW scaling limits.  There, $E_t$ will be taken as the
inverse of a strictly increasing subordinator $D$, defined by \eqref{Etdef}.
For a coupled CTRW, where the jump variable depends on the waiting time, the inner
process $E_t$ and the outer process $Y(x)$ in the scaling limit are dependent.

First we recall briefly the definitions of Hausdorff and packing
dimension. More detailed information together with their applications to
stochastic processes and other areas can be found in Falconer \cite{Fal90},
Kahane \cite{K85}, Taylor \cite{T86} and Xiao \cite{X04}.
For any $\alpha > 0$,  the \emph{$\alpha$-dimensional Hausdorff measure}
of $F \subseteq \R^d$ is defined by
\begin{equation}
\label{Eq:Hausdorff}
\hbox {$s^\alpha$-$m$}(F) = \lim_{\eps \to 0}\
\inf \bigg\{ \sum_i (2 r_i)^\alpha: F \subseteq \bigcup_{i
=1}^{\infty} B(x_i, r_i), \  r_i < \eps \bigg\},
\end{equation}
where $B(x,r)$ denotes the open ball of
radius $r$ centered at $x.$ The sequence of balls satisfying the
two conditions on the right-hand side of (\ref{Eq:Hausdorff}) is
called an \emph{$\eps$-covering} of $F$.  It is well-known that
$s^\alpha$-$m$ is a metric outer measure and every Borel set in
$\R^d$ is $s^\alpha$-$m$ measurable.
%A function $\varphi \in \Phi$ is called an {\it exact Hausdorff
%measure function} for $E$ if $0 < \hbox{$\varphi$-$m$} (E) < \infty$.
The \emph{Hausdorff dimension} of $F$ is defined by
\[
\dimh F = \inf \big\{ \a > 0:\  \hbox {$s^{\a}$-$m$}(F) = 0\big\}=
\sup \big\{ \a > 0:\  \hbox {$s^{\a}$-$m$}(F) = \infty \big\}.
\]
It is easily verified that $\dimh$ satisfies the $\sigma$-stability property:
For any $F_n \subseteq \R^d$, one has
\begin{equation}\label{Eq:stability}
\dimh \bigg(\bigcup_{n=1}^\infty F_n\bigg) = \sup_{n\ge 1} \dimh F_n.
\end{equation}

Similarly to (\ref{Eq:Hausdorff}), the $\a$-dimensional packing measure
of $F \subset \R^d$ is defined as
\[
\hbox {$s^\alpha$-$p$}\,(F) = \inf \biggl\{\sum_n \hbox
{$s^\alpha$-$P$}(F_n) :\ \ F \subseteq \bigcup_n F_n \biggr\},
\]
where $s^\alpha$-$P$ is the set function on subsets of $\R^{d}$
defined by
\[
\hbox {$s^\alpha$-$P$}(F) = \lim_{\ep \to 0}\ \sup \biggl\{ \sum_i
(2 r_i)^\alpha: \overline {B}(x_i, r_i) \ \hbox{are disjoint,} \  x_i
\in F, \ r_i < \ep \biggr\}.
\]
The packing dimension of $F$ is defined by
$
\dimp F = \inf\big\{ \a > 0:\ \ \hbox {$s^{\a}$-$p$}\,(F) = 0
\big\}.
$
It can be verified that $\dimp$ also satisfies the $\sigma$-stability property
analogous to (\ref{Eq:stability}).
%It is well known that  $0 \le \dimh F \le \dimp F\le d$ for every
%set $F\subseteq \R^d$.

The packing dimension can also be defined through the upper
box-counting dimension.
For any $\eps > 0$ and any bounded set $F \subseteq
{\R^d},$ let $N(F, \eps)$ be the smallest number of balls of
radius $\eps$ needed to cover $F$. The upper box-counting
dimension of $F$ is defined as
\begin{equation}\label{Eq:dimm}
\dimm F = \limsup_{\eps \to 0} \frac{\log N(F, \eps)} {- \log \eps}.
\end{equation}
Tricot \cite{Tricot} proved that the packing dimension of $F$ can be obtained
from $\dimm$ by
\begin{equation}\label{Def:dimp}
\dimp F = \inf \bigg\{\sup_{n}\, \dimm F_n: \ \ F \subseteq
\bigcup_{n=1}^\infty F_n\bigg\},
\end{equation}
see also Falconer \cite[p.45]{Fal90}. It is well known
that for every (bounded) set $F\subseteq \R^d$,
\begin{equation}\label{Eq:dim-rel}
0 \le \dimh F \le \dimp F \le \dimm F \le d.
\end{equation}

The following theorem determines the Hausdorff and packing dimension of the range
$X([0, 1]) =\{X(t): t \in [ 0, 1]\}$ in terms of the range of $Y$.

\begin{theorem}\label{Th:Hdim}
Let $X= \{X(t), t \ge 0\}$ be the iterated process with values in $\R^d$
defined by $X(t) = Y(E_t)$, where the processes $Y$ and $E$ satisfy the
aforementioned conditions. If $E_1 > 0$ a.s. and there exist constants $c_1$
and $c_2$ such that for all constants $0 < a  < \infty$
\begin{equation}\label{Eq:Y-Con}
\dimh Y([0, a]) = c_1, \quad  \dimp Y([0, a]) = c_2\qquad \hbox{a.s.},
\end{equation}
then  almost surely
\begin{equation}\label{Eq:Ran-1}
\dimh X([0, 1]) =  c_1\quad \hbox{ and }\ \dimp X([0, 1]) =  c_2.
\end{equation}
\end{theorem}

\begin{proof}\ Since the process $t \mapsto E_t$ is non-decreasing and continuous,
the range $E([0, 1])$ is the random interval $[0, E_1]$. Hence $X([0, 1]) =
Y([0, E_1])$.

It follows from the $\sigma$-stability of $\dimh$ and (\ref{Eq:Y-Con}) that
$\dimh Y([0, \infty)) = c_1$ a.s. Hence $\dimh X([0, 1]) \le c_1$
almost surely. On the other hand, (\ref{Eq:Y-Con}) implies
\begin{equation}\label{Eq:Y-Con0}
\P\Big\{\dimh   Y([0, q]) = c_1, \
\forall  q \in \Q_+\Big\} = 1,
\end{equation}
where $\Q_+$ denotes the set of positive rational numbers. Since $E_1 > 0$
almost surely, we see that there is an event $\Omega'\subset \Omega$ with
$\P(\Omega') = 1$ such that for every $\omega \in \Omega'$ we have
$E_1(\omega) > 0$ and $\dimh Y([0, q], \omega)
= c_1$ for all $q \in \Q_+$. Since for every $\omega \in \Omega'$
there is a $q \in \Q_+$ such that $0 < q <
E_1(\omega)$, we derive that
$$ \dimh X([0, 1], \omega) = \dimh Y([0, E_1(\omega)], \omega)
\ge \dimh Y([0, q], \omega) = c_1.$$
Combining the upper and lower bounds for $\dimh X([0, 1])$ yields the
first equation in (\ref{Eq:Ran-1}). The proof of the second equation
in (\ref{Eq:Ran-1}) is very similar and is omitted.
\end{proof}
%For every fixed $\omega_2 \in \Omega_2$, $X([0, 1]) \subset Y(\R_+)$ for
%every $\omega_1 \in \Omega_1$. Hence for every $\omega_2 \in \Omega_2$,
%\[
%\dimh X([0, 1]) \le \dimh Y(\R_+) = c_1,\qquad \hbox{$\P_1$-a.s.},
%\]
%where the last equality follows from (\ref{Eq:Y-Con}) and the $\sigma$-stability of
%$\dimh$. By Fubini's theorem, we derive $\dimh X([0, 1]) \le \dimh Y(\R_+) = c_1$
%$\P$-almost surely.

Applying Theorem \ref{Th:Hdim} to the space-time process $x\mapsto
(x, Y(x))$ with values in $\R^{d+1}$, one obtains immediately the
following corollary.

\begin{corollary}\label{Co:1}
Let $X= \{X(t), t \ge 0\}$ be the iterated process with values in $\R^d$
as in Theorem \ref{Th:Hdim}.
If  $E_1 > 0$ a.s. and there exist constants $c_3$ and $c_4$ such that for all
constants $0 < a  < \infty$
\begin{equation}\label{Eq:Y-Con1}
\dimh {\rm Gr}Y([0, a]) = c_3\ \  \hbox{ and }\ \ \dimp {\rm Gr}Y([0, a]) = c_4, \quad a.s.,
\end{equation}
then
\begin{equation}\label{Eq:Graph}
\dimh \big\{(E_t, Y(E_t)): t \in [0, 1]\big\} = \dimh {\rm Gr}Y([0, 1]),
\qquad a.s.
\end{equation}
and
\begin{equation}\label{Eq:Graph-p}
\dimp \big\{(E_t, Y(E_t)): t \in [0, 1]\big\} = \dimp {\rm Gr}Y([0, 1]),
\qquad a.s.
\end{equation}
\end{corollary}

The random set in the left hand side of (\ref{Eq:Graph}) may be interesting,
but it is quite different than the graph of $X$.
In order to determine the Hausdorff and packing dimension of the graph set of $X$,
we will make use of the $\R^{d+1}$-valued process $Z = \{Z(x), x\ge 0\}$
defined on the probability space $(\Omega, {\mathcal F}, \P)$ by
\begin{equation}\label{Eq:Z}
Z(x)= (D(x), Y(x)), \qquad \forall x \ge 0,
\end{equation}
where $D= \{D(x), x \ge 0\}$ is defined by
\begin{equation}\label{Eq:D}
D(x) = \inf\big\{t>0: E_t > x\big\}.
\end{equation}
Since $t \mapsto E_t$ is nondecreasing and continuous, it can be verified
that the function $x \mapsto D(x)$ is strictly increasing and
right continuous, thus can have at most countably many jumps. Moreover,
one can verify that $D(E_t) \ge t$ for all $t \ge 0$ and $E_{D(x)} = x$
for all $x \ge 0$.

\begin{theorem}\label{Th:Hdim2}
Let $X= \{X(t), t \ge 0\}$ be the iterated process with values in $\R^d$
as in Theorem \ref{Th:Hdim}, and let $Z = \{Z(x), x\ge 0\}$ be the $\R^{d+1}$-valued
process defined by (\ref{Eq:Z}) and (\ref{Eq:D}). If $E_1 > 0$ a.s. and there exist
constants $c_5$ and $c_6$ such that for all constants $0 < a  < \infty$
\begin{equation}\label{Eq:Y-Con2}
\dimh Z([0, a]) = c_5 \ \ \hbox{ and }\ \ \dimp Z([0, a]) = c_6 \quad \hbox{a.s.},
\end{equation}
then
\begin{equation}\label{Eq:Graph0}
\dimh {\rm Gr}X([0, 1]) =  \max\big\{1, \dimh Z([0, 1])\big\}, \qquad \hbox{a.s.}
\end{equation}
and
\begin{equation}\label{Eq:Graph0-p}
\dimp {\rm Gr}X([0, 1]) =  \max\big\{1, \dimp Z([0, 1])\big\}, \qquad \hbox{a.s.}
\end{equation}
\end{theorem}

\begin{proof} We only prove (\ref{Eq:Graph0}), and the proof of (\ref{Eq:Graph0-p})
is similar. The sample function $x \mapsto D(x)$ is a.s.\ strictly increasing
and we can write the unit interval [0, 1] in the state space of $D$ as
\begin{equation}\label{Eq:Dec1}
[0, 1] = D([0, E_1))\cup \bigcup_{i=1}^\infty I_i,
\end{equation}
where for each $i \ge 1$, $I_i$ is a subintervals on which $E_t$ is a constant.
Using $D$ we can express $I_i= [D(x_{i}-), D(x_i))$, which is the gap corresponding to
the jumping site $x_i$ of $D$, except in the case when $x_i = E_1$. In the latter
case, $I_i = [D(x_{i}-), 1]$.

Notice that $I_i$ ($i \ge 1$) are disjoint intervals and
\[
E_t = E_s \ \hbox{ if and only if }\ s, t \in I_i\, \hbox{ for some }\ i \ge 1.
\]
Thus, over each interval $I_i$, the graph of $X$ is a horizontal line segment.
%The set $\{t \in (0, 1): E_t \hbox{ is a constant in a neighborhood of } t\}$
%is the union of jump intervals of the subordinator $D$.
More precisely, we can decompose the graph set of $X$ as
\begin{equation}\label{Dec2}
\begin{split}
{\rm Gr}X([0, 1]) &= \big\{(t, Y(E_t)): t \in [0, 1]\big\}\\
&= \big\{(t, Y(E_t)): t \in D([0, E_1))\big\} \cup
\bigcup_{i=1}^\infty\big\{(t, Y(E_t)): t \in I_i\big\}.
\end{split}
\end{equation}
Hence, by the $\sigma$-stability of $\dimh$, we have
\begin{equation}\label{Dec3}
\begin{split}
\dimh {\rm Gr}X([0, 1]) &= \max\big\{1, \dimh \big\{(t, Y(E_t)): t \in D([0, E_1))\big\}\big\}.
\end{split}
\end{equation}
On the other hand, every $t \in D([0, E_1))$ can be written as $t = D(x)$ for some
$0 \le x < E_1$ and $E_t = E_{D(x)} = x$, we see that
\begin{equation}\label{Dec4}
\big\{(t, Y(E_t)): t \in D([0, E_1)) \big\}
= \big\{(D(x), Y(x)): x \in [0, E_1))\big\}, \quad
\hbox{a.s.}
\end{equation}
It follows from (\ref{Eq:Y-Con2}) that
\begin{equation}\label{Dec5}
\P\Big\{\omega: \dimh \big\{(D(x, \omega), Y(x, \omega)): x \in [0, q])\big\}
= c_5, \quad \forall   q \in \Q_+\Big\} = 1.
\end{equation}
Combining this with the assumption that $E_1(\omega) > 0$ almost surely, we can find an
event $\Omega_2''$ such that  $\P (\Omega_2'') = 1$ and
for every $\omega \in \Omega_2''$  we derive from (\ref{Dec5}) that
\begin{equation}\label{Dec6}
\dimh \big\{(D(x, \omega), Y(x, \omega)): x \in [0, E_1(\omega))\big\} = c_5,
\end{equation}
since $q_1<E_1(\omega_2)<q_2$ for some $q_1,q_2 \in \Q_+$, and $U\subseteq V$
implies $\dimh(U)\leq\dimh(V)$. Combining (\ref{Dec4}) and (\ref{Dec6}) yields
\begin{equation}\label{Dec7}
\dimh \big\{(t, Y(E_t)): t \in D([0, E_1))\big\} = c_5, \quad
\hbox{a.s.}
\end{equation}
Therefore,  (\ref{Eq:Graph0}) follows from (\ref{Dec3}) and  (\ref{Dec7}).
\end{proof}

Many self-similar processes $Y$ with stationary increments satisfy
(\ref{Eq:Y-Con}) and (\ref{Eq:Y-Con2}), hence Theorems
\ref{Th:Hdim} and \ref{Th:Hdim2} have wide applicability.
To apply the above theorems to CTRW scaling limits,
we now take $E_t$ to be the inverse of a subordinator $D= \{D(x), x \ge 0\}$.
We assume that $D$ has no drift, $D(0) = 0$ and its Laplace transform is given by
$${\mathbb E}[e^{-s D(x)}] = e^{-x\psi_D(s)},$$
where the Laplace exponent
\begin{equation}\label{psiD2}
\psi_D(s)=\int_0^\infty(1-e^{-s y})\nu_D(dy).
\end{equation}
We also assume that the L\'evy measure $\nu_D$ of $D$ satisfies
$ \nu_D(0,\infty)=\infty$,
so that the sample function $x \mapsto D(x)$ is a.s.\ strictly increasing.

Let $E_t$ denote the inverse of $D$ defined by \eqref{Etdef}.
Since the sample function of $D$ is strictly increasing, we see that the function
$t \mapsto E_t$ is almost surely continuous and nondecreasing. Moreover,
$\P\big\{E_1 > 0\big\} = 1$.

\begin{cor}\label{Cor:Hdim}
Let $X= \{X(t), t \ge 0\}$ be the iterated process with values in $\R^d$
as in Theorem \ref{Th:Hdim}, where $E_t$ is the inverse \eqref{Etdef} of a strictly
increasing subordinator $D$ with $D(0) = 0$. Then the conclusions of
Theorem \ref{Th:Hdim}, Corollary \ref{Co:1}, and Theorem \ref{Th:Hdim2} hold.
\end{cor}

%Theorem 3.1 in \cite{M-S-triangular} implies that $E_t$ has a Lebesgue density
%\[
%g(t,x)=\int_0^t \nu_D(t-y,\infty) P_{D_x}(dy)
%\]
%for any $t>0$.  It follows from Theorem III.4 of Bertoin \cite{Bertoin96}
%that, for every fixed $t > 0$, $D(E_t) > t$ a.s.  On the other hand,
%since the sample function $x \to D(x)$ is strictly increasing, we have for all $x \ge 0$ that
%$E_{D(x)} = x$ a.s. Then the arguments for Theorem \ref{Th:Hdim}, Corollary \ref{Co:1},
%and Theorem \ref{Th:Hdim2} extend immediately.

\section{Continuous Time Random Walk Limits}

In the following, we compute the Hausdorff and packing dimensions of
the range and graph of the sample path of scaling limits of continuous
time random walks.

\subsection{CTRW with iid jumps: The uncoupled case}
Consider a CTRW whose iid waiting times $\{W_n, n \ge 1\}$ belong to the
domain of attraction of the positive $\beta$-stable random variable $D(1)$,
and whose iid jumps $\{J_n, n \ge 1\}$ belong to the strict domain of
attraction of the $d$-dimensional stable random vector $Y(1)$. We assume
that $\{W_n\}$ and $\{J_n\}$ are independent; that is, the CTRW is
uncoupled.

It follows from Theorem 4.2 in Meerschaert and Scheffler \cite{MS04}
that the scaling limit of this CTRW is a time-changed process $X(t)=
Y(E_t)$, where $E_t$ is the inverse \eqref{Etdef} of a $\beta$-stable
subordinator $D$. Since $D$ is self-similar with index $1/\beta$, its
inverse $E$ is self-similar with index $\beta$.  Since $Y$ is independent
of $E$, the CTRW scaling limit $X$ is self-similar with index $\beta/\alpha$.

\begin{prop}\label{prop:Levy}
Let $X=\{Y(E_t), t \ge 0\}$, where $Y=\{Y(x):\ x\geq 0\}$ is a
stable L\'evy motion of index $\alpha \in (0, 2]$ with values
in $\R^d$ and $E_t$ is the inverse of a stable subordinator
of index $0<\beta<1$, independent of $Y$. Then
\begin{equation}\label{Eq:Range1}
\dimh X([0, 1]) = \dimp X([0, 1])= \min\{d, \alpha\}, \qquad \hbox{a.s.}
\end{equation}
and
\begin{equation}\label{Eq:G1}
\begin{split}
\dimh {\rm Gr}X([0, 1]) &= \dimp {\rm Gr}X([0, 1])\\
&=\left\{\begin{array}{ll}
 \max\{1, \alpha\}\qquad \ &\hbox{ if } \alpha \le d,\\
 1+\beta(1 - \frac 1 \alpha)  &\hbox{ if }  \alpha > d =1,\\
\end{array}
\right. \qquad \hbox{a.s.}
\end{split}
\end{equation}
\end{prop}

\begin{proof}
The result (\ref{Eq:Range1}) follows from Theorem \ref{Th:Hdim} (or Corollary
\ref{Cor:Hdim})  and the known results on the Hausdorff and packing dimension
of the range of the stable L\'evy process $Y$. The former is due to Blumenthal
and Getoor \cite{BG60a, BG60b}, and the latter is due to Pruitt and Taylor \cite{PT96}.

In order to prove (\ref{Eq:G1}), we first recall from Pruitt and
Taylor \cite{PT69} their result on the Hausdorff dimension of the range
of the L\'evy process $Z(x) = (D(x), Y(x))$ with independent stable
components: for any constant $a> 0$,
\begin{equation}\label{Eq:G1a}
\dimh Z([0, a]) =\left\{\begin{array}{ll}
\beta \qquad \ &\hbox{ if }   \alpha \le \beta,\\
 \alpha\qquad \ &\hbox{ if } \beta <  \alpha \le d,\\
 1+\beta(1 - \frac 1 \alpha)  &\hbox{ if }  \alpha > d =1,\\
\end{array}
\right. \qquad \hbox{a.s.}
\end{equation}
Theorem 3.2 in Meerschaert and Xiao \cite{MX05} (see also Khoshnevisan and Xiao \cite{KX08}
for more general results) shows that $\dimp Z([0, a])$ also equals
the right hand side of (\ref{Eq:G1a}). Therefore, (\ref{Eq:G1}) follows from the
above and Theorem \ref{Th:Hdim2}.
\end{proof}

%In the above, the CTRWs are uncoupled. The proof of Meerschaert and Xiao \cite{MX05}
%for results for operator stable L\'evy processes, which can be applied for
%the coupled case.

If the CTRW jumps $(J_n)$ have finite second moments, then the limiting process is
$X(t) = B(E_t)$, where $B$ is a Brownian motion, and Proposition \ref{prop:Levy}
with $\alpha = 2$ gives
\begin{equation}\label{Eq:G1b}
\dimh {\rm Gr}X([0, 1]) = \dimp {\rm Gr}X([0, 1]) =\left\{\begin{array}{ll}
1+ \frac \beta 2  \qquad \ &\hbox{ if }  d =1,\\
 2 \ &\hbox{ if } d \ge 2,
\end{array}
\right. \qquad \hbox{a.s.}
\end{equation}

The Hausdorff dimension of the  graph of a stable L\'evy process $Y$ in $\R^d$ was
determined by Blumenthal and Getoor \cite{BG62} when $d=1$ and $Y$ is symmetric, by
Jain  and Pruitt \cite{JP68} when $Y$ is transient (i.e. $d > \alpha$) and
by Pruitt and Taylor \cite{PT69} in general. The packing dimension of the graph of
$Y$ was determined by Rezakhanlou and Taylor \cite{RT88}. Combining their results with
Corollary \ref{Co:1}, we obtain
\begin{equation}\label{Eq:G1d}
\begin{split}
\dimh \big\{(E_t, Y(E_t)): t \in [0, 1]\big\} &= \dimp \big\{(E_t, Y(E_t)): t \in [0, 1]\big\}\\
&= \left\{\begin{array}{ll}
 \max\{1, \alpha\}\qquad \ &\hbox{ if } \alpha \le d,\\
 2 - \frac 1 \alpha  &\hbox{ if }  \alpha > d =1.
\end{array}
\right. \qquad \hbox{a.s.}
\end{split}
\end{equation}
Clearly, this is different from (\ref{Eq:G1}). Moreover, we notice that the results
(\ref{Eq:Range1}) and (\ref{Eq:G1d}) do not depend on $\beta$, because
the set $\{E_t(\omega):t \in [0, 1]\}$ is a.s.\ a closed interval,
so the range dimension is the same after the time change.

\subsection{CTRW with iid jumps: The coupled case}

In some applications, it is natural to consider a coupled CTRW where
$\{(J_n,W_n), n \ge 1\}$ are iid, but the jump $J_n$ depends on the preceding waiting
time $W_n$.  We can also extend the results of the last section to certain
coupled CTRW limits.  In the coupled case, the CTRW $S(N_t)$ has scaling limit $Y(E_t-)$ and the
so-called oracle CTRW $S(N_t+1)$ has scaling limit $Y(E_t)$, see \cite{HenryStraka,OCTRW}.
If $Y,D$ are independent, they have a.s.\ no simultaneous jumps, and the two limit processes are the same.
The proof of Theorem \ref{Th:Hdim2} extends immediately to the process $Y(E_t-)$ in this case, with the same
dimension results.  This is because the graphs of  $Z(x)= (D(x), Y(x))$ and $Z'(x)= (D(x), Y(x-))$
have the same Hausdorff and packing dimension, as they differ by at most a countable number of discrete points.
In the following, we discuss examples for $Y(E_t)$, with the understanding that the same dimension results hold for $Y(E_t-)$.

The simplest case is $W_n=J_n$, in which case $X(t)=D(E_t)$.
This process is self-similar with index $1$, see for example Becker-Kern et
al.\ \cite{coupleCTRW}.
% The argument for Proposition \ref{prop:Levy} still applies in this case, to show that
It follows from Theorems \ref{Th:Hdim}, \ref{Th:Hdim2} and the fact that for
any constant $a > 0$,
$$\dimh D([0, a]) = \dimh \{(D(x), D(x)): x \in [0, a]\} = \beta, \quad \hbox{a.s.}$$
that
\begin{equation}\label{Eq:coupledim}
\dimh X([0, 1])= \dimp X([0, 1]) =\beta, \quad a.s.
\end{equation}
and
\begin{equation}\label{Eq:coupledim-g}
\dimh {\rm Gr}X([0, 1])= \dimp {\rm Gr}X([0, 1])= 1, \quad a.s.
\end{equation}
We should also mention that by applying the ``uniform'' Hausdorff and packing
dimension results for the $\beta$-stable subordinator $D$ (see Perkins
and Taylor \cite{PT87}), which states that almost surely
$$ \dimh D(F) = \beta \dimh F\ \hbox{ and } \dimp D(F) = \beta \dimp F \
\hbox{ for all Borel sets } F \subseteq \R,$$
we obtain \eqref{Eq:coupledim} directly by choosing $F = [0, E_1)$.
%The second equation in (\ref{Eq:coupledim}) follows from (\ref{Dec3}), (\ref{Dec4}) and the fact
%that $\dimh \{(D(x), D(x)): x \in [0, E_1)\} = \beta$ almost surely.
%Here the sample path $X([0, 1])$ is a L\'evy dust laid out along the line $x=t$.

Shlesinger et al.\ \cite{SKW} consider a CTRW where the waiting times $W_n\geq 0$
are iid with the $\beta$-stable random variable $D$ and $\E(e^{- s D}) = e^{-s^{\beta}}$
and, conditional on $W_n=t$, the jump $J_n$ is normal with mean zero and variance $2t$.
Then $J_n$ is symmetric stable with index $\alpha=2\beta$.  This model was applied to stock market prices
by Meerschaert and Scalas \cite{coupleEcon}. Becker-Kern et al.\
\cite{coupleCTRW} show that the CTRW limit is $X(t)=Y(E_t)$ ($t \ge 0$), where $Y$ is a
real-valued stable L\'evy process with index $\alpha = 2\beta$ and $E_t$ is the
inverse of a $\beta$-stable subordinator.  Then $X(t)$ is self-similar with
index $1/2$, the same as Brownian motion. However, the Hausdorff dimensions of the
range and graph of $X$ are completely different than those for Brownian motion.

Note that here $E_t$ is not independent of $Y(t)$.
Theorem \ref{Th:Hdim} gives that $\dimh X([0, 1])$ $= \min\{1, 2\beta\}$ a.s.
To determine the Hausdorff dimension of the graph of $X(t)$, we first verify that the
Fourier-Laplace transform of $(D(1), Y(1))$ is
\[
\begin{split}
\E\Big(e^{i\xi Y(1)-\eta D(1)}\Big) &= \E\Big[e^{-\eta D(1)} \E\big(e^{i\xi Y(1)}|D(1)\big)\Big]\\
&= \E\Big(e^{-(\eta+\xi^2) D(1)}\Big) = e^{- (\eta +\xi^2)^\beta}.
\end{split}
\]
It follows that the L\'evy process $Z(x) = (D(x), Y(x))$ is operator
stable \cite{RVbook} with the unique exponent
\begin{equation}\label{JordanBlock2}
C=\begin{pmatrix} \beta^{-1}&0\\ 0& (2\beta)^{-1}\end{pmatrix},
\end{equation}
which has eigenvalues $\beta^{-1}$ and $(2\beta)^{-1}$.

By applying Theorem 3.2 from Meerschaert and Xiao \cite{MX05}, we derive that
for any $a > 0$,
\begin{equation}\label{Eq:G1c}
\dimh Z([0, a]) = \dimp Z([0, a]) =\left\{\begin{array}{ll}
2\beta \qquad \ &\hbox{ if }   2 \beta \le 1,\\
\frac 1 2 +\beta   &\hbox{ if }  2 \beta > 1,\\
\end{array}
\right. \qquad \hbox{a.s.}
\end{equation}
Consequently, we use Theorem \ref{Th:Hdim2} to derive
\begin{equation}\label{Eq:G1de}
\dimh {\rm Gr}X([0, 1])= \dimp {\rm Gr}X([0, 1])=\max\{1,\beta+\tfrac 12\},
\quad \hbox{a.s.},
\end{equation}
which is quite different from the corresponding result \eqref{Eq:G1b}
in the uncoupled case.

\subsection{CTRW with iid jumps: Triangular array limits}

Proposition \ref{prop:Levy} and (\ref{Eq:G1de}) rely on the Hausdorff and packing
dimension results for sample functions of stable or, more generally, operator
stable L\'evy processes. The Hausdorff dimensions and potential theoretic properties
of general L\'evy processes have been studied by several authors (\cite{Pruitt69,
khoshnevisan-xiao,KXZ})
and the packing dimension results have been proved by Khoshnevisan and Xiao \cite{KX08}
and Khoshnevisan, Schilling and Xiao \cite{KSX10}. These results are useful for studying
fractal properties of the CTRW limits under more general settings such as triangular array
schemes.

In particular, for any L\'evy process $Z = \{Z(x), x \ge 0\}$ with values
in $\R^p$ and characteristic exponent $\Phi$ (i.e., $\E(e^{i \l \xi, Z(x)\r})
= e^{-x \Phi(\xi)}$), Corollary 1.8 in \cite{KXZ} shows that for any $a > 0$,
\begin{equation}\label{Eq:KX03}
\dimh Z([0, a]) = \sup\Bigg\{\gamma < p: \int_{\{\xi\in \R^p:\ \|\xi\|\ge 1\}}
\Re\bigg(\frac{1} {1 + \Phi(\xi)}\bigg)
\frac{d \xi}{\|\xi\|^\gamma} < \infty\Bigg\}, \ \ \hbox{a.s.}
\end{equation}
where $\|\cdot\|$ is the Euclidean norm on $\R^p$. On the other hand, Theorem 1.1
in Khoshnevisan and Xiao \cite{KX08} shows that for any $a > 0$,
\begin{equation}\label{Eq:KX08}
\dimp Z([0, a]) =  \sup\left\{\gamma \ge 0:\ \liminf_{r\to 0^+}
        \frac{W(r)}{r^\gamma}  =0\right\},  \ \hbox{a.s.},
\end{equation}
where, for all $r > 0$, $W(r)$ is defined by
\[
W(r) = \int_{\R^p} \Re\bigg(\frac 1 {1+\Phi(\xi/r)}\bigg) \prod_{j=1}^p\frac{1}
{1 +\xi_j^2} \, d\xi.
\]
(More precisely, Theorem 1.1 in \cite{KX08} is proved for $a=1$, but its proof works for
arbitrary $a > 0$. Another way to get (\ref{Eq:KX08}) from  Thorem 1.1 in \cite{KX08}
is to use the stationarity of increments of $Z$.)

By combining (\ref{Eq:KX03}) and (\ref{Eq:KX08}) with Theorems \ref{Th:Hdim} and
\ref{Th:Hdim2} we extend the results in the previous sections to more general
time-changed processes.

\begin{prop}\label{prop:Levy-G}
Let $X=\{Y(E_t), t \ge 0\}$, where $Y=\{Y(x):\ x\geq 0\}$ is a L\'evy process with values
in $\R^d$ and characteristic exponent $\psi$ and let $E_t$ be the inverse of a
subordinator $D = \{D(x), x \ge 0\}$ with characteristic exponent $\sigma$. If $Z
=\big\{(D(x), Y(x)), x \ge 0\big\}$ is a L\'evy process in $\R^{1+d}$ and its characteristic
exponent $\Phi$ satisfies
\begin{equation}\label{Eq:PhiCom}
K^{-1}\, \Re\bigg(\frac{1} {1 + \si(\eta) + \psi(\xi)}\bigg) \le \Re\bigg(\frac{1}
{1 + \Phi(\eta, \xi)}\bigg) \le K\, \Re\bigg(\frac{1} {1 + \si(\eta) + \psi(\xi)}\bigg)
\end{equation}
for all $(\eta, \xi )\in \R^{1+d}$ with $|\eta|+\|\xi\| $ large, where $K\ge 1$
is a constant. Then almost surely,
\[%begin{equation}\label{Eq:Range1e}
\dimh X([0, 1]) = \sup\Bigg\{\gamma < d: \int_{\{\xi\in \R^d:\ \|\xi\|\ge 1\}}
\Re\bigg(\frac{1} {1 + \psi(\xi)}\bigg)
\frac{d \xi}{\|\xi\|^\gamma} < \infty\Bigg\}
\]%end{equation}
and $\dimh {\rm Gr}X([0, 1]) = \max \{1, \chi\}$ almost surely, where
\[
\chi = \sup\Bigg\{\gamma < 1+d: \int_{\{|\eta|+\|\xi\|\ge 1\}}
\Re\bigg(\frac{1} {1 + \sigma(\eta)+\psi(\xi)}\bigg)
\frac{d \eta d \xi}{(|\eta|+ \|\xi\|)^\gamma} < \infty\Bigg\}.
\]%end{equation}
\end{prop}

The packing dimensions of $X([0, 1])$ and ${\rm Gr}X([0, 1])$ are given as follows,
which may be different from the Hausdorff dimensions given in Proposition
\ref{prop:Levy-G}.
\begin{prop}\label{prop:Levy-Gp}
Let $X=\{Y(E_t), t \ge 0\}$ be the same as in Proposition \ref{prop:Levy-G}.
If $Z =\big\{(D(x), Y(x)), x \ge 0\big\}$ is a L\'evy process in
$\R^{1+d}$ and its characteristic exponent $\Phi$ satisfies (\ref{Eq:PhiCom}),
then
\[%begin{equation}\label{Eq:Range1e}
\dimp X([0, 1]) = \sup\left\{\gamma \ge 0:\ \liminf_{r\to 0^+}
        \frac{W(r)}{r^\gamma}  =0\right\},  \quad \hbox{ a.s.},
\]%end{equation}
where $W(r)$ is defined by
\[
W(r) = \int_{\R^d} \Re\bigg(\frac 1 {1+\psi(\xi/r)}\bigg) \prod_{j=1}^d\frac{1}
{1 +\xi_j^2} \, d\xi,
\]
and $
\dimp {\rm Gr}X([0, 1]) = \max \{1, \chi'\}$ almost surely, where
\[
\chi' = \sup\left\{\gamma \ge 0:\ \liminf_{r\to 0^+}
        \frac{\widetilde{W}(r)}{r^\gamma}  =0\right\}
\]%end{equation}
and where
\[
\widetilde{W}(r) = \int_{\R^{1+d}} \Re\bigg(\frac 1 {1+\sigma(\eta/r)+\psi(\xi/r)}\bigg)
\frac{1}
{1 +\eta^2} \prod_{j=1}^d\frac{1}
{1 +\xi_j^2} \, d\eta\,d\xi.
\]
\end{prop}

Next we consider a generalized CTRW limit, as in \cite{M-S-triangular},
obtained by using a triangular array scheme.  This limit can be applied as a stochastic
model for ultraslow diffusion (cf. \cite{M-S-ultra, CGS}).

At each scale $c>0$ we are given iid waiting times $(W_n^{c})$ and iid jumps $(J_n^{c})$.
Assume the waiting times and jumps form triangular arrays whose row sums converge in distribution.
More specifically, let $S^{c}(n)=J_1^{c}+\cdots+J_n^{c}$ and $T^{c}(n)=W_1^{c}+\cdots+W_n^{c}$,
we require that
$S^{c}(cu)\Rightarrow Y(t)$ and $T^{c}(cu)\Rightarrow D(t)$ as $c\to\infty$, where the limits
$Y(t)$ and $D(t)$ are independent L\'evy processes.
Letting $N^{c}_t=\max\{n\geq 0:T^{c}(n)\leq t\}$, the CTRW scaling limit $S^{c}(N_t^{c})
\Rightarrow Y(E_t)$ \cite[Theorem 2.1]{M-S-triangular}.

A power law mixture model for waiting times was proposed in \cite{M-S-ultra}:  Take
an iid sequence of random variables $\{B_i\}$ with $0<B_i<1$ and assume
$\P\{W_i^{c}>u|B_i=\beta\}=c^{-1}u^{-\beta}$ for $u\geq c^{-1/\beta}$, so that
the waiting times are power laws conditional on the mixing variables. The waiting
time process $T^{c}(cu)\Rightarrow D(t)$, which is a subordinator with Laplace
transform ${\mathbb E}[e^{-s D(t)}]=e^{-t\psi_D(s)}$, where \eqref{psiD2} holds.
The L\'evy measure of $D$ is given by
\begin{equation}\label{psiWdef}
\nu_D(t,\infty)=\int_0^1 t^{-\beta}\mu(d\beta),
\end{equation}
where $\mu$ is the distribution of the mixing variable \cite[Theorem 3.4 and Remark 5.1]{M-S-ultra}.
A computation \cite[Eq.\ (3.18)]{M-S-ultra} using
$\int_0^\infty(1-e^{-s t})\beta
t^{-\beta-1}dt=\Gamma(1-\beta)s^\beta$ shows that
\begin{equation}\begin{split}\label{psiW}
\psi_D(s)
%&= \int_0^\infty(1-e^{-s t})\phi_W(dt)\\
%&= \int_0^\infty(1-e^{-s t})f(t)dt\\
%&= \int_0^\infty \int_0^1 (1-e^{-s t}) \beta t^{-\beta-1}\mu(d\beta)dt\\
%&= \int_0^1 \int_0^\infty (1-e^{-s t}) \beta t^{-\beta-1}dt\mu(d\beta)\\
&= \int_0^1  s^\beta \Gamma(1-\beta) \mu(d\beta) .
\end{split}\end{equation}
Then $c^{-1}N^{c}_t\Rightarrow E_t$ the inverse subordinator \cite[Theorem 3.10]{M-S-ultra}.

Now we take $\mu(d\beta)=\sum_{k=1}^n d_k^{\beta_k}(\Gamma(1-\beta_k))^{-1}\delta_{\beta_k}(d\beta)$,
where $0<\beta_1<\beta_2<\cdots <\beta_n<1$ are constants and $\delta_a$ is the unit mass at $a$. In
this case, the subordinator is $D(t)=\sum_{k=1}^n d_k D_k(t)$, which is a mixture of
independent $\beta_k$-stable subordinators $D_k(t)$ ($k=1,\cdots, n$). See Chechkin et al.\
\cite{CGS} for some applications of such CTRW and its scaling limit  $X(t) = Y(E_t)$.

In order to apply Propositions \ref{prop:Levy-G} and \ref{prop:Levy-Gp} to establish
Hausdorff and packing dimension results for the above time-changed process $X$,
we will make use of the following technical result.

\begin{lemma}\label{inverse-mixture-lemma}
Let $D(x)=\sum_{k=1}^n d_k D_k(x)$, where $d_k>0$ are constants and $D_k(x)$ are
independent stable subordinators of index $\beta_k$ and $0<\beta_1<\beta_2<\cdots <\beta_n<1$.
Let $Y = \{Y(x), x \ge 0\}$ be a strictly stable L\'evy motion of index $\alpha \in (0, 2]$
with values in $\R^d$. We assume $D$ and $Y$ are independent and let $\Phi$ be the
characteristic exponent of the L\'evy process $Z(x) = (D(x), Y(x))$. Then for all $(\eta,
\xi)\in \R^{1+d}$  that satisfies $|\eta| + \|\xi\| >1$
\begin{equation}\label{Eq:Exp1}
\frac{K^{-1}}{|\eta|^{\beta_n}+\|\xi\|^\alpha}
\leq \Re\bigg(\frac 1{1+\Phi(\eta,\xi)}\bigg)\leq \frac{K}{ |\eta|^{\beta_n}+ \|\xi\|^\alpha},
\end{equation}
where $K \ge 1$ is a constant which may depend on $n, \alpha$, $\beta_k, d_k$ for $k=1,2,\cdots, n$.
\end{lemma}

\begin{proof}\,
For simplicity, we assume that the characteristic exponent of $Y$
is $\psi(\xi) = \|\xi\|^\alpha$. Then the L\'evy process $Z(x)=(D(x), Y(x))$ has
characteristic exponent
\begin{equation}
\begin{split}
\Phi(\eta, \xi)&=\sum_{k=1}^n \big(-id_k\eta\big)^{\beta_k}+\|\xi\|^{\alpha}\\
&=\sum_{k=1}^n |d_k\eta|^{\beta_k}[\cos(\pi\beta_k/2)-i \sin(\pi\beta_k/2) ]
+\|\xi\|^{\alpha}\\
&=: f(\eta, \xi) -i g(\eta).
\end{split}
\end{equation}
Since $\beta_k \in (0, 1)$, we have $f(\eta, \xi) \ge 0$
for all $\eta, \xi \in \R^{1+d}$. Moreover, $0 \le g(\eta) \le K f(\eta, \xi)$
for some constant $K > 0$. Hence
\[
\frac{1}{(1+K^2)(1+f(\eta, \xi))} \le \Re\bigg(\frac 1{1+\Phi(\eta, \xi)}\bigg) \leq
 \frac{1}{1+f(\eta, \xi)}.
\]
From here it is elementary to verify (\ref{Eq:Exp1}).
\end{proof}

By using Lemma \ref{inverse-mixture-lemma}, Propositions \ref{prop:Levy-G} and
\ref{prop:Levy-Gp} we derive the following proposition. Since the proof is similar to
that of Proposition 4.1 in \cite{MX05} (see also Proposition 7.7 in
\cite{khoshnevisan-xiao}), we omit the details.

\begin{prop}\label{prop:Levy-1}
Let $X=\{Y(E_t), t \ge 0\}$, where $Y=\{Y(x):\ x\geq 0\}$ is a strictly
stable L\'evy motion of index $\alpha \in (0, 2]$ with values
in $\R^d$ and $E_t$ is the inverse of a subordinator $D(t)=\sum_{j=1}^n d_k D_k(t)$,
where $d_k>0$ and $D_k(t)$ are independent stable subordinators of index $\beta_k$
and  $0<\beta_1<\beta_2<\cdots \beta_n<1$. Suppose also that $E$ is
independent of $Y$. Then
\begin{equation}\label{Eq:Range1f}
\dimh X([0, 1]) = \dimp X([0, 1])= \min\{d, \alpha\}, \qquad \hbox{a.s.}
\end{equation}
and
\begin{equation}\label{Eq:G1f}
\begin{split}
\dimh {\rm Gr}X([0, 1]) &= \dimp {\rm Gr}X([0, 1])\\
&= \left\{\begin{array}{ll}
 \max\{1, \alpha\}\qquad \ &\hbox{ if } \alpha \le d,\\
 1+\beta_n(1 - \frac 1 \alpha)  &\hbox{ if }  \alpha > d =1,\\
\end{array}
\right. \qquad \hbox{a.s.}
\end{split}
\end{equation}
\end{prop}

%\begin{remark}
%Let $D$ be an arbitrary strictly increasing L\'evy subordinator without drift, as in Corollary
%\ref{Cor:Hdim}.  Since the inverse process $E_t$ is continuous and nondecreasing,
%\eqref{Eq:Range1} still holds for the process $X=\{Y(E_t)\}$.  Theorem III.15 in
%Bertoin \cite{Bertoin96} shows that, for every $s>0$, we have
%$$
%\beta=\dimh D([0,s])=\sup \{\gamma>0: \ \lim_{s\to\infty}
%s^{-\gamma}\psi_D(s)=\infty\} \ \ \text{a.s.}
%$$
%It may also be possible to extend Proposition \ref{prop:Levy-1} to this case, using the
%index formula of Khoshnevisan et al.\ \cite{KXZ}.
%\end{remark}

\subsection{CTRW with correlated jumps}
Now consider an uncoupled CTRW whose jumps $\{J_n\}$ form a correlated sequence
of random variables, and whose waiting times $\{W_n\}$ are iid and belong to
the domain of attraction of a positive $\beta$-stable random variable $D(1)$.

We further assume that $\{J_n\}$ and $\{W_n\}$ are independent. In this case,
Meerschaert, et al.\ \cite{MNX09} show that, under certain conditions
on the correlation structure of $\{J_n\}$, the CTRW scaling limit is the
$(H\beta)$-self-similar process $X=\{Y(E_t): t\geq 0\}$,
where $Y$ is a fractional Brownian motion with index $H \in (0, 1)$,
and $E_t$ is the inverse of a $\beta$-stable subordinator $D$ which is
independent of $Y$.

The following proposition determines the Hausdorff and packing dimension of
the sample path of $X$.

\begin{prop}\label{Prop:FBM}
Let $X=\{Y(E_t), t \ge 0\}$, where $Y$ is a fractional
Brownian motion with values in $\R^d$ of index $H \in (0, 1)$ and
$E_t$ is the the inverse of a $\beta$-stable subordinator
$D$ which is independent of $Y$. Then
\begin{equation}\label{Eq:Range2}
\dimh X([0, 1]) = \dimp X([0, 1])= \min \Big\{d,\, \frac 1 H\Big\}, \qquad \hbox{a.s.}
\end{equation}
and
\begin{equation}\label{Eq:G2}
\begin{split}
\dimh {\rm Gr}X([0, 1]) &= \dimp {\rm Gr}X([0, 1])\\
& = \left\{
\begin{array}{ll}
\frac 1 H \qquad  &\hbox{ if } 1 \le H d,\\
\beta + (1 - H \beta) d  &\hbox{ if } 1 > H d,
\end{array}
\right. \qquad \hbox{a.s.}
\end{split}
\end{equation}
\end{prop}

\begin{proof}\, Eq. (\ref{Eq:Range2}) follows from Theorem \ref{Th:Hdim}
and the well known result on Hausdorff and packing dimension of the
range of fractional Brownian motion (see, e.g., Chapter 18 of \cite{K85}).
In order to prove (\ref{Eq:G2}), by Theorem \ref{Th:Hdim2} it is
sufficient to prove that for the $\R^{d+1}$-valued process
$Z = \{Z(x), x \ge 0\}$ defined by $Z(x) = (D(x), Y(x)), x \ge 0$
and for any constant $a > 0$, we have
\begin{equation}\label{Eq:Z2}
\begin{split}
\dimh Z([0, a]) &= \dimp Z([0, a])\\
& = \left\{
\begin{array}{ll}
\frac 1 H \qquad  &\hbox{ if } 1 \le H d,\\
\beta + (1 - H \beta) d  &\hbox{ if } 1 > H d,
\end{array}
\right. \qquad \hbox{a.s.}
\end{split}
\end{equation}
Thanks to (\ref{Eq:dim-rel}), we can divide the proof of \eqref{Eq:Z2}
into proving the upper bound for $\dimp Z([0, a])$ and  the lower bound
for $\dimh Z([0, a])$ separately. These are given as Lemmas \ref{Lem:Up} and
\ref{Lem:Lw} below.
\end{proof}

\begin{lemma}\label{Lem:Up}
Let the assumptions of Proposition \ref{Prop:FBM} hold and let $a > 0$
be a constant. Then
\begin{equation}\label{Eq:Z3}
\dimp Z([0, a]) \le \left\{
\begin{array}{ll}
\frac 1 H \qquad  &\hbox{ if } 1 \le H d,\\
\beta + (1 - H \beta) d  &\hbox{ if } 1 > H d,
\end{array}
\right. \qquad \hbox{a.s.}
\end{equation}
\end{lemma}

In order to prove Lemma \ref{Lem:Up}, we will make use of  the fact that for
every $\eps> 0$ the function $Y(x)$ ($ 0 \le x \le a$) satisfies the
uniform H\"older condition of order $H - \eps$ and the following Lemma
\ref{Lem:LX98}, which is an immediate consequence of Lemma 3.2 in
Liu and Xiao \cite{LX98}. It can also be derived from Lemma 6.1 in
Pruitt and Taylor \cite{PT69}.

Let $c_7 > 0$ be a fixed constant. A
collection $\Lambda(b)$ of intervals of length $b$ in $\R$ is called $c_7$-nested
if no interval of length $b$ in $\R$ can intersect more than $c_7$ intervals
of $\Lambda(b)$. Note that for each integer $n \ge 1$, the collection of dyadic
intervals $I_{n,j}=[j/2^n,(j+1)/2^n]$ is $c_7$-nested with $c_7 = 3$.

\begin{lemma}\label{Lem:LX98}
Let $\{D(x), x \ge 0\}$ be a $\beta$-stable subordinator and let $\Lambda(b)$ be
a $c_7$-nested family. Denote by $M_u(b, s)$ the number of intervals in
$\Lambda(b)$ which intersect $D([u, u+s])$. Then there exists a positive
constant $c_8$ such that for all $u \ge 0$ and all $0 < b^\beta \le s$,
\begin{equation}\label{Eq:D1}
\E\Bigl(M_u(b, s)\Bigr) \le c_8\, s b^{-\beta}.
\end{equation}
If one takes \ $b = s \le 1$, then we  have
\begin{equation}\label{Eq:D2}
\E\Bigl(M_u( a, s)\Bigr) \le  c_8.
\end{equation}
\end{lemma}

Now we are ready to prove Lemma \ref{Lem:Up}.

{\it Proof of Lemma \ref{Lem:Up}.} \, The proof is based on a moment argument.
%which is a modification of that in Xiao and Lin \cite{XiaoLin}.
We divide the interval $[0, a]$ into $(\lfloor a \rfloor +1)2^n$ dyadic
intervals $I_{n, j}$ of length $2^{-n}$.

%Fix an arbitrary constant $\eps \in(0, H)$ and an integer $n \ge 2$,
First we construct a covering of the range $Z([0, a])$ by using balls
in $\R^{d+1}$ of radius $2^{-Hn}$ as follows.
%Let ${\mathcal C}_n$ be the collection of cubes in $\R^{d+1}$ of the form
%$\prod_{i=1}^{d+1}[\frac{k_i}{2^{(H-\eps)n}}, \frac{k_i+1}{2^{(H-\eps)n}}]$,
%where $k_1, \ldots, k_{d+1}$ are integers.
Define $t_{n, j}=j/2^n$ so that for each $I_{n, j} = [t_{n, j}, t_{n, j}+2^{-n}]$,
the image $Y(I_{n, j})$ is contained in a ball in $\R^d$
of radius $\sup_{s \in I_{n, j}} \|Y(s) - Y(t_{n, j})\|$ and can be covered by
at most
\begin{equation}\label{Eq:Nn}
N_{n, j}= c_9\,\bigg(\frac{ \sup_{s \in I_{n, j}} \|Y(s) - Y(t_{n, j})\|} {2^{-Hn}}\bigg)^d
\end{equation}
balls of radius $2^{-Hn}$. By the self-similarity and stationarity of increments
of $Y$, we have
\begin{equation}\label{Eq:Nn2}
\begin{split}
\E\big(N_{n, j}\big)&= c_9\,2^{Hdn}
\E\Big[\big(\sup_{s \in I_{n, j}} \|Y(s) - Y(t_{n, j})\|\big)^d\Big] \\
&= c_9 \E\Big[\big(\sup_{s \in [0, 1]} \|Y(s)\|\big)^d\Big] :=c_{10} < \infty,
\end{split}
\end{equation}
where the last inequality follows from the well known tail probability for
the supremum of Gaussian processes (e.g., Fernique's inequality).

In order to get a covering for $D(I_{n, i})$, let $\Gamma(2^{-n})$ be the
collection of dyadic intervals of order $n$ in $\R_+$. Let $M_{n, j}$ be the
number of dyadic intervals in $\Gamma(2^{-n})$ which intersect $D(I_{n, j})$.
Applying (\ref{Eq:D2}) in Lemma \ref{Lem:LX98} with $b_n = s_n = 2^{-n}$, we
obtain that
\begin{equation}\label{Eq:Mn}
\E(M_{n, j}) \le c_8, \quad \forall \ 1 \le j \le (\lfloor a \rfloor +1)2^n.
\end{equation}
Since $2^{-n} < 2^{-Hn}$, we see that $Z(I_{n,j}) = \{(D(x), Y(x)): x \in I_{n, j}\}$
can be covered by at most $M_{n, j}N_{n, j}$ balls in $\R^{d+1}$ of radius $2^{-Hn}$.
Denote by $N\big(Z([0, a]), 2^{-Hn}\big)$ the smallest number of balls in $\R^{d+1}$
of radius $2^{-Hn}$ that cover $Z([0, a])$, then
\[
N\big(Z([0, a]), 2^{-Hn}\big) \le \sum_{j=1}^{(\lfloor a \rfloor +1) 2^{n}}
 M_{n, j}N_{n, j}.
\]
It follows from (\ref{Eq:Nn2}), (\ref{Eq:Mn}) and the independence of
$Y$ and $D$ that
\[
\E\Big[N\big(Z([0, a]), 2^{-Hn}\big)\Big]
\le (\lfloor a \rfloor +1) c_8 c_{10}\, 2^n.
\]
Hence, for any $\eps > 0$,
\[
\P\Big\{ N\big(Z([0, a]), 2^{-Hn}\big) \ge (\lfloor a \rfloor +1) c_8 c_{10}\, 2^{n (1+\eps)}\Big\} \le
2^{-n\eps}.
\]
It follows from the Borel-Cantelli lemma that almost surely
$$N\big(Z([0, a]), 2^{-Hn}\big) < (\lfloor a \rfloor +1) c_8 c_{10}\, 2^{n (1+\eps)}$$
for all $n$ large enough. This and (\ref{Eq:dimm}) imply that $\dimm Z([0, a]) \le (1+\eps)/H$ a.s.
Since $\eps>0$ is arbitrary, we obtain from the above and (\ref{Eq:dim-rel}) that
$\dimp Z([0, a]) \le 1/H$ almost surely.

Next we construct a covering for the range $Z([0, a])$ by using balls
in $\R^{d+1}$ of radius $2^{-n/\beta}$. Let $\Gamma(2^{-n/\beta})$ be the
collection of intervals in $\R_+$ of the form $I'_{n,k}=[\frac{k} {2^{n/\beta}},
\frac{k+1}{2^{n/\beta}}]$, where $k$ is an integer. Then the class $\Gamma(2^{-n/\beta})$
is $3$-nested. Let $M'_{n, j}$ be the
number of intervals in $\Gamma(2^{-n/\beta})$ that intersect $D(I_{n, j})$.
By Lemma \ref{Lem:LX98} with $b_n = 2^{-n/\beta}$ and $s_n = 2^{-n}$,
we derive $\E(M'_{n, j}) \le c_8$. Thus $D(I_{n, j})$ can almost surely be covered by
$M'_{n, j}$ intervals of length $2^{-n/\beta}$ from $\Gamma(2^{-n/\beta})$.

On the other hand, the image $Y(I_{n, j})$ can be covered by at most
\[
N'_{n, j}= c_9\,\bigg(\frac{ \sup_{s \in I_{n, j}}\|Y(s) - Y(t_{n, j})\|}
{2^{-n/\beta}}\bigg)^d
\]
balls of radius $2^{-n/\beta}$, where $t_{n, j}=j/2^n$, and then similar to (\ref{Eq:Nn2})
we derive
\begin{equation}\label{Eq:Nn3}
\E\big(N'_{n, j}\big)= c_{10} 2^{n(\frac 1 {\beta} - H)d}.
\end{equation}
Denote by $N\big(Z([0, a]), 2^{-n/\beta}\big)$ the smallest number of balls in $\R^{d+1}$
of radius $2^{-n/\beta}$ that cover $Z([0, a])$, then
\[
N\big(Z([0, a]), 2^{-n/\beta}\big) \le \sum_{j=1}^{(\lfloor a \rfloor +1) 2^{n}}
 M'_{n, j}N'_{n, j}.
\]
By (\ref{Eq:Nn3}) and the independence of $Y$ and $D$ we have
\[
\E\Big[N\big(Z([0, a]), 2^{-n/\beta}\big)\Big]
\le (\lfloor a \rfloor +1) c_8 c_{10}\, 2^{n(1+(\frac 1 {\beta} - H)d)}.
\]
Hence, for any $\eps > 0$, the Borel-Cantelli Lemma implies that a.s.
\[
N\big(Z([0, a]), 2^{-n/\beta}\big) < (\lfloor a \rfloor +1) c_8 c_{10}\,
2^{n(1+(\frac 1 {\beta} - H )d +\eps)}
\]
for all $n$ large enough. This and (\ref{Eq:dimm}) imply that $\dimm Z([0, a])
\le \beta + (1 - \beta H)d +\beta \eps$ almost surely which,  in turn, implies
$\dimp Z([0, a]) \le \beta + (1 - \beta H)d$ a.s.

Combining the above we have
\[
\dimp Z([0, a]) \le \min\Big\{\frac 1 H, \beta + (1 - \beta H)d\Big\}
\quad \hbox{a.s.}
\]
This proves (\ref{Eq:Z3}). \qed

%\begin{remark}
%The covering argument in Lemma \ref{Lem:LX98} should extend to a mixture of the stable subordinators, see the proof of Theorem 15 in Chapter III in %Bertoin \cite{Bertoin96}.  It should also be possible to extend Lemma \ref{Lem:LX98} to certain subordinated Markov processes.  Such processes were %considered in Baeumer et al. \cite{bmn-07}.  They can occur as the scaling limit of CTRW when the jump distribution depends on the current location in %space of a random walker, see Kolokoltsov \cite{K}.
%\end{remark}

\begin{lemma}\label{Lem:Lw}
Under the assumptions of Proposition \ref{Prop:FBM},  we have
\begin{equation}\label{Eq:Z5}
\dimh Z([0, 1]) \ge \left\{
\begin{array}{ll}
\frac 1 H \qquad  &\hbox{ if } 1 \le H d,\\
\beta + (1 - H \beta) d  &\hbox{ if } 1 > H d,
\end{array}
\right. \qquad \hbox{a.s.}
\end{equation}
\end{lemma}

\begin{proof}\, Since the projection of $Z([0, 1])$ into $\R^d$ is $Y([0, 1])$ and
$\dimh Y([0, 1]) =  \frac 1 H$ a.s.\ when $1 \le H d$. This implies the first
inequality in \eqref{Eq:Z5}.

To prove the inequality in \eqref{Eq:Z5} for the case $1 > Hd$, by Frostman's theorem
(cf. \cite[p.133]{K85}) along with the inequality
$$\|Z(x)-Z(y)\|
\geq \frac 1 2 \big[|D(x) - D(y)| + \|Y(x) - Y(y)\| \big],$$
it is sufficient to prove that for every constant $\gamma \in (0, \beta + (1 - H \beta) d)$,
we have
\begin{equation}\label{Eq:En1}
\E\int_0^a\int_0^a \frac {dx\, dy } {\big[|D(x) - D(y)| + \|Y(x) - Y(y)\|\big]^\gamma}< \infty.
\end{equation}
Since $1 > Hd$, we have $\beta + (1 - H \beta) d > d$. We only need to verify
(\ref{Eq:En1}) for every $\gamma \in (d,  \beta + (1 - H \beta) d)$.

For this purpose, we will make use of the following easily verifiable fact
(see, e.g., Kahane \cite[p.279]{K85}): If
$\Xi$ is a standard normal vector in $\R^d$, then there is a finite constant $c_{11}> 0$
such at for any constants $\ga > d$  and $\rho \ge 0$,
\[
\E\bigg[\frac 1 {\big(\rho + \|\Xi\|\big)^{\ga}}\bigg] \le
c_{11}\, \rho^{-(\ga -d)}.
\]

Fix $x, y \in [0, a]$ such that $x\ne y$. We use $\E_1$ to denote the conditional
expectation given the subordinator $D$,  apply the above fact with
$\rho = |D(x) - D(y)||x-y|^{-H}$ and use the self-similarity of $D$ to derive
\begin{equation}\label{Eq:En2}
\begin{split}
&\E\bigg(\frac 1 {\big[|D(x) - D(y)| + \|Y(x) - Y(y)\|\big]^\gamma}\bigg)\\
&= |x-y|^{-H\gamma} \E\Bigg[\E_1\bigg(\frac 1 {(\rho + \|\Xi\|)^\gamma}
\bigg)\Bigg]\\
&\le c_{11} |x-y|^{-H\gamma}\, \E\Bigg[\frac {|x-y|^{H(\gamma - d)}}
{|D(x)-D(y)|^{\gamma-d}}\Bigg]\\
&=  c_{12} \frac 1 {|x-y|^{H d + (\ga-d)/\beta}},
\end{split}
\end{equation}
where the last equality follows from the $1/\beta$-self-similarity of
$D$ and the constant $c_{12} = c_{11} \E \big(D(1)^{-(\gamma - d)}\big)$. Recall
from Hawkes \cite[Lemma 1]{H71} that, as $r \to 0+$,
\[
\P(D(1) \le r) \sim c_{13}
r^{\beta/(2(1- \beta))}\exp\bigg(- (1-\beta)\beta^{\beta/(1-\beta)}\,
r^{-\beta/(1-\beta)}\bigg),
\]
where $c_{13}= \big[ 2\pi (1-\beta) \beta^{\beta/(2(1-\beta))}\big]^{-1/2}$.
We verify easily $c_{12} <\infty.$

It follows from Fubini's theorem and (\ref{Eq:En2}) that
\begin{equation}\label{Eq:En3}
\begin{split}
&\E\int_0^a\int_0^a \frac {dx\, dy } {\big[|D(x) - D(y)| + \|Y(x) - Y(y)\|\big]^\gamma}\\
&\le c_{12} \int_0^a\int_0^a \frac {dx\, dy } {|x-y|^{H d + (\ga-d)/\beta}} < \infty,
\end{split}
\end{equation}
the last integral is convergent because $H d + (\ga-d)/\beta< 1$. This proves
(\ref{Eq:En1}) and thus the lemma.
\end{proof}

\begin{remark}
Other iterated processes can arise as scaling limits of CTRW with dependent
and/or heavy-tailed jumps or waiting times. For example, the process $Y$ can
be taken as a linear fractional stable motion, see \cite{MNX09}. Our main
theorems in Section 2 are applicable to these self-similar processes too.
However, the problems for determining the Hausdorff dimensions of the range
and graph sets of the processes $Y$ and $Z(x) = (D(x), Y(x))$ have not been
satisfactorily solved, see \cite{SX10,XiaoLin} for partial solutions.
\end{remark}

\bigskip

\end{document}